\RequirePackage{fix-cm}
\documentclass[smallextended]{svjour3}       
\smartqed  
\usepackage{graphicx}
\usepackage{epstopdf}
\usepackage{dsfont,amsmath,amsfonts, amssymb, mathrsfs}
\usepackage{color}
\usepackage[colorlinks=true,linkcolor=blue]{hyperref}

\newcommand{\Cn}{ \mathds{C}^n }
\newcommand{\CO} {\Cn\setminus\Omega}

\newcommand{\Om}{\Omega = \left\{z\in\Cn : \rho(z)<0 \right\}}
\newcommand{\dom}{ {\partial\Omega} }
\newcommand{\OO}{\Omega_\eps\setminus\Omega}

\renewcommand{\O}{\Omega}
\newcommand{\eps}{\varepsilon}

\newcommand{\C}{\mathbb{C}}
\newcommand{\R}{\mathbb{R}}
\newcommand{\N}{\mathbb{N}}

\newcommand{\RE}{\text{\rm{Re}}}

\newcommand{\pr}{ \text{\rm{pr}}}

\renewcommand{\P}{\mathcal{P}}
\newcommand{\f}{{\mathbf{f}}}

\newcommand{\cn}{\frac{1}{(2\pi i)^n}}
\newcommand{\dbar}{\bar\partial}

\newcommand{\suml}{\sum\limits}

\newcommand{\grad}[1]{ \text{\rm{grad}}\ {#1} }
\newcommand{\dist}[2]{ \text{\rm{dist}} (#1,\ #2)}

\newcommand{\supp}[1]{ \text{\rm{supp}}\ {#1} }

\newcommand{\scp}[2]{ \left\langle #1,\ #2  \right\rangle }
\newcommand{\norm}[1]{\left\lVert #1 \right\rVert}
\newcommand{\abs}[1]{\left\lvert #1 \right\rvert}

\begin{document}

\title{Constructive description of analytic Besov spaces in strictly pseudoconvex domains.\thanks{The work is supported by Russian Science Foundation Grant 19-11-00058.}
}

\titlerunning{Analytic Besov spaces}        

\author{Aleksandr Rotkevich
}


\institute{A. Rotkevich \at
              St. Petersburg University \\
              Tel.: +7-952-2433484\\
              \email{rotkevichas@gmail.com}
}

\date{Received: date / Accepted: date}

\maketitle

\begin{abstract}
We use the method of pseudoanalytic continuation to obtain a characterization of spaces of holomorphic functions with boundary values in Besov spaces in terms of polynomial approximations.
\end{abstract}

\section{Introduction}
\label{intro}
The constructive description of the smoothness of functions in terms of polynomial approximations is a classical problem that started by D. Jackson and S.Bernstein results. In 1984 E. M. Dynkin gave a constructive characterization of holomorphic Besov spaces in simply connected domains in $\C$ with "good" boundary. We continue the research (see \cite{R13,R18_1,R18_2,Sh89}) devoted to the description of spaces of functions of several complex variables and consider spaces of holomorphic functions on strictly pseudoconvex domains, which boundary values are in Besov space. The main obstacle for polynomial approximations in strictly pseudoconvex is that polynomials are not always dense in space of holomorphic functions that are continuous up to the boundary. We restrict our consideration to strictly pseudoconvex Runge domains.
\begin{definition}
The domain $\O\subset\Cn$ is Runge domain if for every function $f$ holomorphic in the neighbourhood of $\overline\O$ and every $\eps>0$ there exist a polynomial $P$ for which $\abs{f(z)-P(z)}<\eps,\ z\in\overline\O.$ 
\end{definition}
The condition that $\O$ is Runge is necessary and sufficient to obtain our results (see \cite{Sh89}).

\section{Main notations and definitions.}
\label{notations}

Let $\Cn$ be the space of $n$ complex variables, $n\geq 1,$ $z =
(z_1,\ldots, z_n),\ z_j = x_j + i y_j;$
\begin{equation*}
\partial_j f =\frac{\partial f}{\partial z_j} = \frac{1}{2}\left( \frac{\partial f}{\partial x_j} - i \frac{\partial f}{\partial y_j}\right), \quad \bar\partial_j f = \frac{\partial f}{\partial\bar{z}_j} = \frac{1}{2}\left( \frac{\partial f}{\partial x_j} + i \frac{\partial f}{\partial y_j}\right);
\end{equation*}

\[
\partial f = \suml_{k=1}^{n} \frac{\partial f}{\partial z_k} dz_k,\quad \bar{\partial} f = \suml_{k=1}^{n} \frac{\partial f}{\partial\bar{z}_k} d\bar{z}_k,\quad df=\partial f+  \bar{\partial} f;
\]

\[
\abs{\bar\partial f} = \abs{\bar\partial_1 f} + \ldots+\abs{\bar\partial_n f}
\]
We use the notation 
$
\scp{\eta}{w} = \suml_{k=1}^{n} \eta_k w_k
$
to indicate the action of the differential form $\eta=\suml_{k=1}^{n}\eta_k dz_k$ of type $(1,0)$ on the vector
$w\in\Cn.$

Let $\Om$ be a strictly pseudoconvex domain with a $C^2$-smooth defining function $\rho:\Cn\to\R.$ We also consider a family of domains $\O_t=\left\{z\in\Cn : \rho(z)<t \right\},$ where t is small real parameter, and a bijection $\Phi_t:\dom\to\dom_t$ given by the exponential map of a normal vector field. This allows us to define the projection $\Psi:\O_\eps\setminus\O_{-\eps}\to\dom$ by $\Psi(\xi)=\Phi_{\rho(\xi)}^{-1}(\xi).$

For $\xi\in\dom_r = \left\{\xi\in\Cn : \rho(\xi)=r
\right\}$ we define the tangent space
\[
 T_\xi^{\R} = \left\{ z\in\Cn : \RE \scp{\partial{\rho}(\xi)}{\xi-z} = 0  \right\},
\]
the complex tangent space
\[
T_\xi = \left\{ z\in\Cn : \scp{\partial{\rho}(\xi)}{\xi-z} = 0  \right\}  
\]
and the complex normal vector 
\begin{equation}
 n(\xi)=\abs{\bar\partial\rho(\xi)}^{-1}\left(\bar\partial_1\rho(\xi),\ldots,\bar\partial_n\rho(\xi)\right).
\end{equation}

We denote the space of holomorphic functions as $H(\O)$ and consider
the Hardy space (see~\cite{S76,FS72}) 
\[
H^p(\O):=\left\{f\in
H(\O):\ \norm{f}_{H^p(\O)}=\sup\limits_{-\eps<t<0} \norm{f}_{L^p(\dom_t)} <\infty\right\},
\]
where Lebesgue spaces $L^p(\dom_t)$ are defined by induced on the boundary of $\O_t$ Lebesgue measure $d\sigma_t.$ We also denote $d\sigma = d\sigma_0.$ Notice that every function $f\in H^p(\O)$ has nontangential boundary limit $f^*\in L^p(\dom)$ and $\norm{f^*}_{H^p(\O)}\asymp\norm{f^*}_{L^p(\dom)}.$ 

Throughout this paper we use notations $\lesssim,\ \asymp.$  We write
$f\lesssim g$ if $f\leq c g$ for some constant $c>0,$ that doesn't
depend on main arguments of functions $f$ and $g$ and usually depend
only on the dimension $n$ and the domain $\O.$ Also $f\asymp g$ if $c^{-1}
g\leq f\leq c g$ for some $c>1.$ We denote the
Lebesgue measure in $\Cn$ as~$d\mu.$

\section{Leray-Koppelman formula and pseudoanalytic continuation}
 In the context of the theory of several complex variables there is no canonical reproducing formula, however we can use the Leray theorem that allows us to construct holomorphic reproducing kernels. The following lemma allows us to consider polynomial Leray maps for strictly pseudoconvex Runge domains.
 
 \begin{lemma}[Shirokov \cite{Sh89}]\label{lm:Sh}
Let $\Om$ be a strictly pseudoconvex Runge domain with the $C^2$-smooth defining function $\rho.$ Then there exist functions $v(\xi,z),\  q(\xi,z),\ w_j(\xi,z),$ constants $\eps,\beta,\delta,s>0$ and a domain $G\subset\C$ such that
\begin{enumerate}
 \item $v(\xi,z),\ q(\xi,z),\ w_j(\xi,z)$ are polynomials in $z$ and continuous in $\xi\in{\overline{\Omega}_\eps\setminus\Omega}.$
 \item $v(\xi,z)= \sum\limits_{j=1}^n w_j(\xi,z)(\xi_j-z_j)=\scp{w(\xi,z)}{\xi-z}.$
 \item $w_j(\xi,z)=\partial_j\rho(\xi)+\sum\limits_{k=1}^n P_{kj}(\xi,z)(\xi_k-z_k)$ where $P_{kj}$ are polynomials in $z.$
 \item For $\xi\in{\overline{\Omega}_\eps\setminus\Omega}$ and $z\in\overline{\O}$ 
  \begin{align*}
&  \abs{v(\xi,z)}\geq \rho(\xi)-\rho(z)+\beta\abs{\xi-z}^2,\ \abs{\xi-z}<\delta; \\
&   \abs{v(\xi,z)}\geq s>0,\ \abs{\xi-z}>\delta.
  \end{align*}
 \item For every $\xi\in{\overline{\Omega}_\eps\setminus\Omega}$ and $z\in\O$ the point $\lambda(\xi,z)=v(\xi,z)q(\xi,z)$ lies in the domain $G$ and $G$ is simply connected region with $C^2$-smooth boundary, which is tangent to the $y-$axis at the origin.
\end{enumerate}
\end{lemma}
By Leray-Koppelman formula for every $f\in H^1(\O)$  we have
\begin{equation} \label{eq:CLF}
 f(z) = K f(z) =  \int\limits_{\dom}  \frac{f^*(\xi) \omega(\xi,z)}{\scp{w(\xi,z)}{\xi-z}^n},\ z\in\Omega,
\end{equation}
where $\omega(\xi,z) =\cn w(\xi,z)\wedge \left(\bar\partial_\xi w(\xi,z)\right)^{n-1}.$

\begin{definition} The function $f\in C^1(\O_\eps\setminus\overline\O)$ with $\supp\f\in\O_\eps$ is a pseudoanalytic continuation of function $f\in H^1(\O)$ if nontangential boundary values of $f$ and $\f$ coinside on $\dom.$
\end{definition}
Applying Stoke's theorem to Leray-Koppelman integral and pseudoanalytic continuation we have (see \cite{R18_1} for details)
\begin{equation} \label{CLF_co}
f(z) = \int\limits_{\O_\eps\setminus\O} \frac{ \bar\partial \f(\xi)\wedge \omega(\xi,z)}{\scp{w(\xi,z)}{\xi-z}^n},\ z\in\O,
\end{equation}
since
\[
d_\xi\frac{ \omega(\xi,z)}{\scp{w(\xi,z)}{\xi-z}^n}  = 0,\quad z\in\O,\ \xi\in\OO.
\]
We denote the kernel by $K(\xi,z)=\scp{w(\xi,z)}{\xi-z}^{-n}.$

The function $d(\xi,z)=\abs{v(\xi,z)}=\abs{\scp{w(\xi,z)}{\xi-z}}$ defines on $\dom$ a quasimetric with constant $A>0$ such that for $z,z_0,\xi\in\dom$
\begin{equation}\label{eq:QM}
d(z,z_0)\leq A(d(z,\xi)+d(\xi,z_0)),\ v(\xi,z)\leq Av(z,\xi).\end{equation}
If $B(z,\delta)=\{\xi\in\dom:d(\xi,z)<\delta\}$ is a quasiball with respect to $d$ then $\sigma(B(z,\delta))\lesssim \delta^n.$ Adapting ideas of L. Lanzani and E.M. Stein in \cite{LS13} we have the following estimate.
\begin{lemma}
$d(\xi,z)\asymp d(\Psi(\xi),z)+\rho(\xi),\ \xi\in\OO,\ z\in\dom.$
\end{lemma}

Thus $\mu(V(\xi,\delta))\lesssim\delta^{n+1},$ where $V(\xi,\delta)=\{z\in\OO : d(\xi,z)<\delta\}$ and analogously to \cite{LS13,R13} we have the following classical estimates. 
\begin{lemma}  \label{LerayEst}
 Let $\alpha>0$ and $0<r<\delta<\eps.$ Then
\begin{align*}
& \int\limits_{z\in\dom,\ d(\xi,z)>\delta} \frac{d\sigma(z)}{d(\xi,z)^{n+\alpha}} \lesssim \delta^{-\alpha},\quad \xi\in\OO;\\
& \int\limits_{\xi\in\dom_r,\ d(\xi,z)>\delta}
\frac{d\sigma_r(\xi)}{d(\xi,z)^{n+\alpha}} \lesssim
\delta^{-\alpha},\quad z\in\Omega;\\
& \int\limits_{z\in\dom,\ d(\xi,z)<\delta} \frac{d\sigma(z)}{d(\xi,z)^n} \lesssim 1+ \log{\frac{\delta}{r}},\quad \rho(\xi) = r<\eps;\\
& \int\limits_{\xi\in\dom_r,\ d(\xi,z)<\delta} \frac{d\sigma_r(\xi)}{d(\xi,z)^n} \lesssim 1+ \log{\frac{\delta}{r}},\quad z\in\Omega.
\end{align*}

\end{lemma}

Concluding ideas by N. Shirokov we have the following polynomial approximation of a kernel.

\begin{lemma}[Shirokov \cite{Sh89}] \label{lm:KGlob}
Let $\O$ be a strongly pseudoconvex Runge domain and $\alpha>0.$ Then for every
$m\in\N$ there exists a function $K^{glob}_m(\xi,z)$ which is continuous in $\xi\in\OO,$  polynomial in $z$ with $\deg K^{glob}_m(\xi,\cdot)\lesssim m$ and satisfies the
following properties:

\begin{equation} \label{eq:CLF_approx1}
 \abs{K(\xi,z) - K^{glob}_m(\xi,z)} \lesssim \frac{1}{m^{\alpha}} \frac{1}{d(\xi,z)^{n+\alpha}},\quad d(\xi,z)\geq \frac{1}{m};
\end{equation}

\begin{equation} \label{eq:CLF_approx2}
 \abs{K^{glob}_m(\xi,z)}\lesssim m^n,\quad d(\xi,z)\leq\frac{1}{m}.
\end{equation}

\end{lemma}

\section{Holomorphic Besov spaces}
The main result of this work is the characterization of holomorphic functions with boundary values in Besov spaces in terms of polynomial approximations. We use definition of Besov spaces in terms of polynomial modulus of smoothness, but it coincides with classical definition in case of smooth domains (see~\cite{R13}).

\begin{definition}
Let $K\subset\dom,\ 1\leq p\leq\infty.$ The best approximation of a function $f\in L^p(K)$ by polynomials of degree not greater than $m\in\N$ we denote by
\begin{equation} \label{df:Em}
 E_{m}(f,K)_p = \inf_{\deg P\leq m} \norm{f-P}_{L^p(K)}.
\end{equation}
\end{definition}
Notice that
\begin{equation} \label{eq:E1Ep}
 E_m(f,K)_1 \leq \sigma(K)^{1-\frac{1}{p}} E_m(f,K)_p.
\end{equation}

\begin{definition}
We say that a set $F\subset\dom$ is $t-$fat if there exist such $t>0$ and $\xi\in\dom$ that $B(\xi,t/2)\subset F \subset B(\xi,t).$ We say that a family $\{F_j\}$ of disjoint subsets is a $t-$decomposition of $\dom$ if $\dom=\bigcup F_j$ and every $F_j$ is $t-$fat.
\end{definition}

\begin{definition}\label{df:om_poly}
Let $f\in L^p(\dom),\ 1
\leq p\leq\infty.$ We define the polynomial modulus of smoothness of degree $m\geq 0$ as
\begin{equation}
\omega_m(f,h)_p = \sup\norm{ \left\{E_{m-1}(f,F_j)_p\right\}_{j=1}^N}_{l^p},
\end{equation}
where supremum is taken over all $t-$decompositions
$\left\{F_j\right\}$ of $\dom.$
\end{definition}

\begin{definition} \label{df:Bps_poly}
  Let $1\leq p,q\leq\infty,\ s>0$ and $f\in L^p(\dom).$ We let $f\in B_{pq}^s(\dom)$ if
        $c_{pq}(f) = \norm{ \omega_m(f,t)_p t^{-s-1/q}}_{L^q(0,\eps)} < \infty,$
        where $m\in\N$ and $m>s.$ The definition does not depend on $m>s.$ 
\end{definition}
We define holomorphic Besov space as
\[
A^s_{pq}(\O)=\{f\in H^p(\O): f^*\in B^s_{pq}(\dom)\}
\]
with norm $\norm{f}_{A^s_{pq}} = \norm{f^*}_{L^p(\dom)} + c_{pq}(f^*).$
\section{Local approximations on the boundary.}

In this section we study an interpolation construction of local approximation on fat subsets of $\dom.$

Let $P_0$ be  polynomial in $\R$ of degree $m$ such that
\begin{equation*}
\int\limits_0^1 P_0(x) dx =1,\quad \int\limits_0^1 x^kP_0(x) dx =
0,\ k=1,...,m,
\end{equation*}
and $P_1$ be a polynomial in $\C$ of degree $m$ such that
\begin{equation*}
    \int\limits_{[-1,1]^2} P_1(z) d\mu(z) =1,\quad \int\limits_{[-1,1]^2} z^k \overline{P_1(z)} d\mu(z) = 0,\ k=1,...,m.
\end{equation*}

Without loss of generality we may assume that $\xi=0$ and $T^{\R}_{\xi} = \{z\in\Cn : z_n\in\R\}.$ Then projections $\pr_\xi,\ \pi_\xi$ to the tangent and complex tangent spaces are defined by
\[
\pr_{\xi} (z) = (z',\RE w),\ \pi_\xi(z) = (z',0),\ z=(z',w)\in\C^{n-1}\times\C.
\]
For $u\in\dom,\ h,h_1>0,$ we define a rectangle $\tilde{Q}_u\subset T^{\R}$ by
\[
\tilde{Q}_u= \tilde{Q}_u(h,h_1) = \pr_\xi(u) + \{z\in\Cn : z'\in
[0,h)^{2n-2}\in\C^{n-1},\ w \in[0,h_1)\}.
\]

The projection $\pr_\xi:\dom\to T^{\R}_{\xi}$ is locally invertible and assuming that $h,h_1>0$ and $\abs{\xi-u}$ are small enough we define $Q_u = \pr_\xi^{-1} (\tilde{Q}_u)$ and
\[
u=(u',w_0) =\pr_\xi^{-1} (z',\RE(u_n)),\ (z',w_1(z')) =\pr_\xi^{-1} (z',\RE(u_n)+h_1),
\]
where $(z',0)\in\pi_\xi(Q_u).$

Consider 
a curve $\tilde{\gamma}_{z'}$ that connects points
$(z',w_0)$ and $(z',w_1(z'))$ in $\dom$ 
\[
\tilde{\gamma}_{z'} = \pr_{\xi}^{-1}\{ (z',\RE(u_n)+th_1),\ 0\leq t\leq 1  \}.
\]
Finally we denote by ${\gamma}_{z'}$ the curve in $\C$ which is generated by $n$-th coordinate of $\tilde{\gamma}_{z'}$ and connects points $w_0$ and $w_1(z').$ 

Let $Q'_u = \pi_\xi(Q_u)$ and define a function
\begin{equation*}
P_{Q_u}(z',w) = \frac{1}{h^{2(n-1)}}\bar{P}_{n-1}\left(\frac{z'-u'}{h}\right)
\frac{1}{w_1(z')-w_0}P_0\left(\frac{w-w_0}{w_1(z')-w_0}\right).
\end{equation*}
where $\bar{P}_{n-1}(z') =  \overline{P_1(z_1)\cdot\ldots\cdot P_1(z_{n-1})}.$

Then for every polynomial $T$ with degree not greater $m$ in every variable
\begin{multline*}
\int\limits_{Q'_u} d\mu(z') \int\limits_{\gamma_{z'}} T(z',w)
P_{J_u}(z',w) dw\\ =
 \int\limits_{Q'_u} \frac{d\mu(z')}{h^{2(n-1)}}\bar{P}_{n-1}\left(\frac{z'-u'}{h}\right) \int\limits_{w_0}^{w_1(z')}  P_0\left(\frac{w-w_0}{w_1(z')-w_0}\right)\frac{T(z',w) dw}{w_1(z')-w_0} \\
=\int\limits_{Q'_u}\bar{P}_{n-1}\left(\frac{z'-u'}{h}\right)\frac{d\mu(z')}{h^{2(n-1)}}\int\limits_{0}^1 T(z',w_0+ v (w_1(z')-w_0)) dv \\
=\int\limits_{Q'_u} T(z',w_0)\bar{P}_{n-1}\left(\frac{z'-u'}{h}\right)\frac{d\mu(z')}{h^{2(n-1)}} = T(u',w_0) = T(u).
\end{multline*}

We can replace the integral over a curve $\gamma_{z'}$ that connects points $w_0(z')$ and $w_1(z')$ with the integral over a segment $[w_0(z'),w_1(z')]$ because internal function is holomorphic.

We pass now to the description of operator that provides almost best polynomial approximation 
on a $h-$fat set $K\in\dom$ with $B(\xi,h/2)\subset K\subset B(\xi,h).$ Note that $$\pr_\xi(B(\xi,h))\subset [-c\sqrt{h},c\sqrt{h})^{2(n-1)}\times [-ch,ch)=\tilde{Q}\subset T^{\R}_\xi$$ for some $c>0.$ We cover the rectangle $\tilde{Q}$ by a union of $(m+1)^n$ rectangles $\tilde{Q}^j=\tilde{Q}_{u^j}(2c\sqrt{h}/[\sqrt{m+1}],2ch/(m+1)),$ $u_j\in\dom.$  

Let $Q^j = \pr_\xi^{-1}(\tilde{Q}^j).$ We define $\P_K(f)$ as a polynomial such that
\begin{equation} \label{P_K}
    \P_K(f)(u^j) = \int_{\pi_\xi(Q^j)} d\mu(z') \int_{\gamma_{z'}}
P_{Q^j}(z',w) f(z',w)dw , \ j=0,\ldots, m^n.
\end{equation}

Then $\P_K(T)(u^j)=T(u^j)$ and $\P_K(T)=T$ for every polynomial $T$ with degree not greater than $m$ in every coordinate, moreover,
\begin{equation} \label{P_est}
  \max\limits_{\dist{z}{K} \leq \lambda h} \abs{\P_K(f)(z)} \leq\frac{c(\lambda,m,\O)}{\sigma(K)}\int_K \abs{f}d\sigma,
\end{equation}
    as  $\max\limits_{B(\xi,\lambda h)} \abs{P_{Q_j}(z)} <\frac{c(\lambda,m,\Omega)}{h^n},$
    where the constant $c$ depends on $\lambda,\ m$ и $\Omega,$ and does not depend on a set $K$ and function $f.$ This implies
\[
  \norm{ f- \P_K f }_{L^p(K)}\leq (c+1) E_m(f,\ K)_p,
\]
and the polynomial $\P_K f$ gives almost best polynomial approximation of function $f$.

\begin{note} \label{rm:part_poly}
 Let $f\in H^p(\dom),\ 1
 \leq p
 \leq\infty$ and $\{F_j\}$ be a $h$-decomposition of $\dom.$ Define a function $T_h$ that is equal on $F_j$ to a polynomial $\P_{F_j}f.$ Then
\[
  \omega_m(f,h)_p\asymp \sup \norm{ f-T_h }_{L^p(\dom)},
\]
where the supremum is taken over all such decompositions.
\end{note}

\begin{note}
 The existence and uniqueness of a polynomial $P_K(f)$ needs some explanation. The uniqueness is simple induction by dimension $n.$ For $n=1$ it is obvious. Assume that we know result for dimension $(n-1).$ Let $T$ be a polynomial of degree not greater than $m$ in every variable such that $T(u^j)=0.$ Rearrange points $u^j=u^{(\alpha)}=(u_1^{\alpha_1},\ldots,u_{n-1}^{\alpha_{n-1}},u_n^{\alpha})$ by multiindex $\alpha=(\alpha_1,\ldots,\alpha_n)$ with $\alpha_k=0..m.$ Notice that points $(u_1^{\alpha_1},\ldots,u_{n-1}^{\alpha_{n-1}})$ form a rectangular net in $Q_h\subset T_\xi.$ By induction for every fixed $\alpha_1$ we have $T(u^{\alpha_1}_1,u_2,\ldots,u_n)= 0.$ Thus $T(u_1,\ldots,u_n)\equiv 0$ for every fixed $(u_2,\ldots,u_n)\in\C^{n-1}$ as polynomial of degree not greater than $m$ in $u_1\in\C$ that has $m+1$ roots.
  
 Proceed to a construction. Let $n_{kj}=\overline{u^k-u^j},$ then the polynomial 
 $$\suml_{k=0}^{m} s_k\prod\limits_{j\neq k}\frac{\scp{n_{kj}}{z-u^j}}{\scp{n_{kj}}{u^k-u^j}}$$
 has degree not greater $m$ in every variable and is equal to $s_k$ at $u^k.$
\end{note}

\section{Two methods of pseudoanalytic continuation}
Let $\f$ be a pseudoanalytic continuation of function $f\in H^1(\O)$ and $1\leq p\leq\infty.$ We introduce the following important characteristics of function $\f$
\begin{equation} \label{eq:SP}
 S_p(\f,r) = \norm{ \bar{\partial} \f }_{L^p(\dom_r)},\ r>0.
\end{equation}
In this section we generalize ideas by E.M. Dynkin \cite{D81} to construct pseudoanalytic continuations with some estimates in this value.

\subsection{Pseudoanalytic continuation by global polynomial approximations.\label{Cont_glob}}
By \cite{S76} strict pseudoconvexity of domain $\O$ implies that functions holomorphic in neighbourhood of $\O$ are dense in $H^1(\Omega).$ Also every holomorphic in neighbourhood of $\O$ function can be approximated on $\overline{\O}$ by polynomials since $\O$ is Runge. Thus there exists a polynomial sequence $P_1,P_2,\ldots$ converging to $f^*$ in $L^1(\dom).$ Let 
\begin{equation}
\lambda(z) = \rho(z)^{-1} \abs{P_{2^{m+1}}(z) - P_{2^{m}}(z)},\quad 2^{-m}<\rho(z)\leq 2^{-m+1}.
\end{equation}

\begin{theorem} \label{thm:ContGlob}
 Assume that $\lambda\in L^p(\OO)$ for some $p\geq 1.$ Then there exist a pseudoanalytic continuation $\f$ of the function $f$
 such that
\begin{equation}\label{eq:PAC_la}
\abs{\bar{\partial} \mathbf{f}(z)} \lesssim \lambda(z),\quad
z\in\O_\eps\setminus\overline{\O}.
\end{equation}
\end{theorem}

\begin{proof}
Consider a function $\chi\in C^\infty(0,\infty)$ such that $\chi(t)=1$
for $t\leq \frac{5}{4}$ and $\chi(t)=0$ for $t\geq \frac{7}{4}.$ We let 
for $m\in\N$ 
\[
\mathbf{f}_0(z) =  P_{2^{m}}(z) + \chi(2^m\rho(z)) (P_{2^{m+1}}(z) - P_{2^{m}}(z)),\ 2^{-m}<\rho(z)<2^{-m+1},
\]
and define the continuation of the function $f$ by formula $\f=\chi(2\rho(z)/\eps)\f_0(z).$

Now $\mathbf{f}$ is $C^1$-function on $\Cn\setminus\overline{\Omega}$ and
$\abs{\bar{\partial} \mathbf{f}(z)} \lesssim \lambda(z).$ We define
a function $F_m(z)$ as $F_m(z)= \mathbf{f}(z)$ for $\rho(z)>2^{-m}$
and as $F_m(z) = P_{2^{m+1}}(z)$ for $\rho(z)<2^{-m}.$ The function $F_m$ is smooth and holomorphic in $\Omega_{2^{-m}},$ and
$\abs{\dbar F_m(z)}\lesssim\lambda(z)$ for
$z\in\Cn\setminus\Omega_{2^{-m}}.$ Thus 
 we get
$$P_{2^{m+1}}(z) = F_m(z) = \int\limits_{\O_\eps\setminus\O} \frac{\bar{\partial}F_m(\xi) \wedge\omega(\xi,z)}{\scp{w(\xi,z)}{\xi-z}^n},\ z\in\Omega,$$

We can pass to the limit in this formula by the dominated
convergence theorem; hence, the function~$\f$ is a pseudoanalytic continuation of the
function~$f$. \qed
\end{proof}

\begin{lemma}\label{lm:Pm}
Let $P_{2^m}$ be a polynomial of degree $2^m$ and $1\leq q\leq\infty.$ Then
\begin{equation}
\norm{P_{2^m}}_{L^q(\dom_r)}\lesssim \norm{P_{2^m}}_{L^q(\dom)},\ 2^{-m}\leq r\leq 2^{-m+1}
\end{equation}
and the constant does not depend on $m$.
\end{lemma}
\begin{proof}
Let $\xi\in\dom$ and $n=n(\xi)$ be a complex normal at this point. For some $\delta>0$ and every $w\in \mathbb{CP}^{n-1}$ a set $$\gamma_w=\{u\in\C: (n+\xi+w)u\in\dom,\ \abs{u}<\delta \}$$ is a simple non closed $C^2$-smooth curve. Also curves $\tilde{\gamma}_w=\{u(n+\xi+w):u\in\gamma_w,\ |u|<\delta\}$ cover a neighbourhood of $\xi\in\dom$ in a set $\dom\setminus T_\xi.$

Consider a conformal map $\psi_w:\C\setminus\gamma_w\to\{|v|>1\}$ such that $\psi_w'(\infty)>0.$ The smoothness of $\dom$ implies that the map $\psi_w$ can be continued to a conformal map of closed domains. The smoothness of $\dom$ implies that there exists some neighbourhood $V=V_\xi$ of point $\xi$ in $\Cn$ and constant $c_1,c_2>0$ such that 
\begin{equation}\label{eq}
\gamma_{w,r}=\{u\in\C: u(n+w)\in\dom_r\cap V\}\subset\{u\in\C: c_1r<\psi_w(u)-1<c_2r\}.
\end{equation}
Thus
\[
\abs{\psi_w(u)}^{2^{m+2}}\asymp (1+2^{-m})^{2^m}\asymp 1,\ 2^{-m}\leq \rho(u(n+w))< 2^{-m+1};
\]
\[
\abs{\psi_w(u)}=1,\ u\in\gamma_w.
\]
Consider a function
\begin{equation*}
H_w(u)=P_{2^m}(u(n+w))\psi_w(u)^{-2^{n+2}},
\end{equation*}
holomorphic in $\C\setminus\gamma_w$ such that $\abs{u}H_w(u) \to 0,\ u\to\infty.$ Then
\begin{equation*}
\sup\limits_{r>0}\norm{H_w}_{L^q(\abs{\psi_w}=1+r)}
\lesssim\norm{H_w}_{L^q(\gamma_w)}
\end{equation*}
and
\begin{equation*}
\sup\limits_{r>0}\norm{H_w}_{L^q(\gamma_{w,r})}
\lesssim\norm{H_w}_{L^q(\gamma_w)}.
\end{equation*}
Then  
\begin{equation*}
 \norm{P_{2^m}}_{L^q(\gamma_{w,r})}\lesssim \norm{H_w}_{L^q(\gamma_{w,r})}\lesssim \norm{H_w}_{L^q(\gamma_{w})} =\norm{P_{2^m}}_{L^q(\gamma_{w})}
\end{equation*}
for $2^{-m }\leq r< 2^{-m+1}$
and integrating this estimate by $w\in \mathbb{CP}^{n-1}$ we get
\begin{equation*}
 \norm{P_{2^m}}_{L^q(V\cap\dom_r)}\lesssim \norm{P_{2^m}}_{L^q(\dom\bigcap V)},\ 2^{-m}\leq r<2^{-m+1}.
\end{equation*}

Finally, we choose the finite covering of $\dom$ that also covers a set $\O_{2^{-m+1}}\setminus\O_{2^{-m}}$ and obtain the desired estimate.

\end{proof}
\begin{corollary}\label{cor:ContGlob}
The continuation $\f$ in Theorem~\ref{thm:ContGlob} satisfies an estimate
\begin{equation} \label{eq:SpEn1}
S_p(\f,r) \lesssim 2^m E_{2^m}(f)_p,\  2^{-m}\leq r\leq 2^{-m+1}.
\end{equation}
\end{corollary}

\subsection{Pseudoanalytic continuation by local polynomial approximations}
Let $f\in H^1(\O),\ m>0,\ z\in\OO.$ Denote $J(z)=B(\Psi(z),\rho(z)/10)$ and
\begin{equation}
E(f,z)=E_m(f,J(z))_1.
\end{equation} 

\begin{theorem} \label{thm:LocCont}
 Assume that $p\geq 1$ and 
 \begin{equation}
   E(f,z)\rho(z)^{-(n+1)}\in L^p(\OO).
 \end{equation}
Then $f\in L^p(\dom)$ and there exist such pseudoanalytic continuation $\f$ of function $f$ such that 
 \begin{equation}
 \abs{\dbar\f(z)}\lesssim E(f,z)\rho(z)^{-(n+1)}.
 \end{equation}
\end{theorem}

\begin{proof}
First, we will prove that $f\in L^p(\dom).$ Assume that $J\subset\dom$ is $h-fat$ and let
\[
L(J) =\left\{z\in\OO: c_1 h<d(\xi,z)<c_2 h\ \text{for some}\ \xi\in J\right\}
\] 
where constants $0<c_1<c_2<1$ are chosen such that $J(z)\subset J$ for every $z\in L(J).$
Then $\mu(L(J)) \asymp h^{n+1}$ and $E_m(f,J)_1\leq E(f,z)
$ for $z\in L(J).$

Consider a series of decompositions
$ \dom=\bigcup\limits_{k=1}^{2^{nm}} J_k^m $
into $2^{-m}$-fat sets 
assuming that every succeeding decomposition is obtained by dividing every set of previous to  $2^n$ parts.

We define a function $T_m$ on $\dom$ assuming $T_m(z) = \P_{J_k^m}f(z)$
for $z\in J_k^m,$ where $\P_{J_k^m}$ is a polynomial projector defined in previous section. Then
\[
f = T_{m_0} + \sum\limits_{k=m_0}^{\infty} (T_{m+1}-T_m)
\]
and it is enough to prove that $\sum\limits_{k=m_0}^\infty\norm{T_{m+1}-T_m }_{L^p(\dom)}<\infty. $

For $z\in J^{m+1}_k$ we have
\begin{multline*}
\abs{T_{m+1}(z)-T_m(z)} = \abs{\P_{J_k^{m+1}}f(z) - \P_{J_k^m}f(z) }= \abs{\P_{J_k^{m+1}}\left(f(z) - \P_{J_k^m}f(z)\right) }\\ \lesssim \frac{1}{\sigma(J_k^{m+1})} \int_{J_k^{m+1}} \abs{f- \P_{J_k^m}f} d\sigma
 \lesssim 2^{nm}E(f,J_k^m)_1\\ \lesssim 2^{nm} \inf\{ E(f,z) : z\in L(J_k^m) \}.
\end{multline*}

Thus 
\[
\norm{T_{m+1}-T_m}_{L^p(J_k^{m+1})}\lesssim 2^{-m(n+1)(1-1/p)} \norm{E(f,z)\rho(z)^{-(n+1)}}_{L^p(L(J_k^m))}.
\]
Consider $L_m = \{ z\in\Cn\setminus\Omega :
c_12^{-m}<\dist{z}{\Omega}<c_22^{-m}\}.$ Then $L(J^k_m)\subset L_m$ and
\[
\norm{T_{m+1}-T_m }_{L^p(\dom)}\lesssim 2^{-m(n+1) \left(1-{1}/{p}\right) } \norm{E(f,z)\rho(z)^{-(n+1)}}_{L^p(L_m)}
\]
Finally
\[
 \sum\limits_{m=m_0}^{\infty} \norm{ T_{m+1}-T_m }_{L^p(\dom)} \leq \norm{E(f,z)\rho(z)^{-(n+1)} }_{L^p(\OO)}<\infty.
 \]

Now we will construct a continuation of function $f$. Consider Whitney decomposition in $\OO,$ $\sum\limits_{k=1}^\infty\chi_k =1,\ \chi_k\in C_0^\infty(\OO),\
\abs{\grad{\chi_k}(z)}\lesssim\rho(z)^{-1}.$ Let
$z_k\in\supp\chi_k$ and $J_k = B(\Psi(z_k),c_1\rho(z)).$  Define
\[
\mathbf{f}(z) = \sum\limits_{k=1}^{\infty} \chi_k(z)\P_{J_k}f(z), \quad z\in\Cn\setminus\Omega.
\]
Note that for every polynomial $T(z)$ we have
\[
\mathbf{f}(z) = T(z) + \sum\limits_{k=1}^{\infty} \chi_k(z)(\P_{J_k}f(z) - T(z)), \quad z\in\Cn\setminus\Omega
\]
and
\[
\bar{\partial} \mathbf{f}(z) = \sum\limits_{k=1}^{\infty} (\P_{J_k}f(z) - T(z)) \bar{\partial}\chi_k(z), \quad z\in\Cn\setminus\Omega.
\]
Let $T = \P_{J(z)}f.$ Then $J_k\subset J(z)$ if
$\chi_k(z)\neq 0.$ Thus by estimate (\ref{P_est}) we get
\[
 \abs{\P_{J_k}f(z) - T(z) }= \abs{ \P_{J_k}(f-T)(z) } \lesssim \sigma(J(z))^{-1} E(f,z)\lesssim \rho(z)^{-n} E(f,z).
 \]
Also $\supp\f\subset\O_\eps$ and
\[
 \abs{ \bar{\partial} \mathbf{f}(z) } \lesssim  E(f,z)\rho(z)^{-(n+1)} \in L^p(\OO). 
 \]
This concludes the proof. 
\end{proof}

\begin{corollary}\label{cor:LocCont}
 The continuation $\f$ of function $f$ in Theorem~\ref{thm:LocCont} satisfies an estimate $S_p(\f,r)\lesssim\omega_m(f,h)_p.$
\end{corollary}
\begin{proof}
Notice that by estimate (\ref{eq:E1Ep}) we have
\[
E(f,z) \leq \sigma(J(z))^{1-\frac{1}{p}}E_m(f,J(z))_p
\lesssim \rho(z)^{\left(1-\frac{1}{p}\right)n}E_m(f,J(z))_p,
\]
 where $J(z)=B(\Psi(z),\rho(z)/10).$ Then
\[
 S_p(\f,r) \lesssim r^{-n/p-1}\norm{ E_m(f,J(z))_p}_{L^p(\dom_r)}.
\]
We will estimate the integral on the right side. Quasiballs $J(z)$ cover $\dom$ and we can choose a finite covering
$J_k=\{\xi\in\dom : d(\xi,z_k)<r\}$ such that for every $z\in\dom_r$ the set $J(z)$ is a subset of $J_k$ for some $k$ and every point is covered not more than by $5^{n}$ sets. Consequently
\[
 \norm{E_m(f,J(z))}_{L^p(\dom_r)}\lesssim r^{n/p}\norm{\{ E_m(f,J_k)\}_{k=1}^N}_{l^p} \lesssim r^{n/p} \omega_m(f,10r)_p,
\]
and $S_p(\f,r)\lesssim \omega_m(f,10r)_p/r.$
\end{proof}

\subsection{Construction of local approximation from pseudoanalytic continuation.}
In the proof of Theorem~\ref{thm:ContGlob} we constructed a pseudoanalytic continuation by local approximations of a function. Assume now that we have some pseudoanalytic continuation $\f$ of function $f.$ We will construct a local polynomial approximations of $f$ on $\dom$ in terms of function $\f.$

Let $J\subset B(z_0,h/A)$ for some $z_0\in\dom$ and $h>0,$ where $A>0$ is a constant of quasimetric $d$ defined in (\ref{eq:QM}). We consider a polynomial of degree $\leq Cm$ 
\begin{equation}
P_J(z) = \int\limits_{d(\xi,z_0)>2h} \bar{\partial}\f(\xi)\wedge\omega(\xi,z) K^{loc}_m(\xi,z,z_0),
\end{equation}
where the kernel $K^{loc}_m$ is a polynomial in $z$ of degree $\leq Cm$ and is defined by following equation.
\begin{multline} \label{K_loc}
K(\xi,z) = \frac{1}{v(\xi,z)^n} = \frac{1}{\left(v(\xi,z_0)+\Delta\right)^n} =
\frac{1}{v(\xi,z_0)^n}\left( 1 +
\frac{\Delta}{v(\xi,z_0)}
\right)^{-n}\\
=\frac{1}{v(\xi,z_0)^n}\Bigg( \sum\limits_{k=0}^m C_n^k\frac{\Delta^k}{v(\xi,z_0)^k}
+
O\left(\frac{\abs{\Delta}}{d(\xi,z_0)}\right)^{m+1}\Bigg)\\
= K^{loc}_m(\xi,z,z_0) +
O\left(\frac{\abs{\Delta}^{m+1}}{d(\xi,z_0)^{m+n+1}}\right),
\end{multline}
where $\Delta =v(\xi,z)-v(\xi,z_0)=\scp{w(\xi,z)}{\xi-z}-\scp{w(\xi,z_0)}{\xi-z_0}.$

Note that for $d(\xi,z_0)>2h$ and $d(z,z_0)<h/A$ we have $d(\xi,z)\asymp d(\xi,z_0)$ and 
\begin{multline*}
\abs{\Delta}\leq\abs{\scp{w(\xi,z)-w(\xi,z_0)}{\xi-z}} + \abs{\scp{w(\xi,z_0)}{z-z_0}}\\
\lesssim\abs{z-z_0}\abs{\xi-z}+\abs{\scp{w(\xi,z_0)-w(z,z_0)}{z-z_0}}
+\abs{\scp{w(z,z_0)}{z-z_0}}\\
\lesssim \abs{z-z_0}\abs{\xi-z}+\abs{w(\xi,z_0)-w(z,z_0)}\abs{z-z_0}+d(z,z_0)\\
\lesssim d(z,z_0)^{1/2}d(\xi,z)^{1/2}+d(z,z_0)\lesssim d(z,z_0)^{1/2}d(\xi,z)^{1/2},
\end{multline*}
because
\begin{multline*}
\abs{w(\xi,z_0)-w(z,z_0)}\lesssim\abs{\partial\rho(\xi)-\partial\rho(z)}\\
+\sum\limits_{k,j=1}^n\left( \abs{P_{jk}(z_0,\xi)}\abs{\xi-z_0}+\abs{P_{jk}(z_0,z)}\abs{z-z_0}\right)\\
\lesssim\abs{\xi-z}+\abs{\xi-z_0}+\abs{z-z_0}\lesssim d(\xi,z)^{1/2}.
\end{multline*}
Consequently,
\begin{equation} \label{eq:ker_cont_m}
\abs{ K(\xi,z)- K^{loc}_m(\xi,z,z_0) }\lesssim
\frac{d(z,z_0)^{\frac{m+1}{2}}}{d(\xi,z)^{n+\frac{m+1}{2}}},\ d(\xi,z_0)
>2h,\ d(z,z_0)<h/A.
\end{equation}
This implies that for $z\in J$
\begin{multline} \label{eq:PolyByCont}
 \abs{ f(z) - P_J(z) }\\
  \lesssim \int\limits_{d(\xi,z_0)<2h}\abs{\bar{\partial}\mathbf{f}(\xi)} \frac{d\mu(\xi)}{d(\xi,z)^n} +
 \int\limits_{d(\xi,z_0)>2h}\abs{\bar{\partial}\mathbf{f}(\xi)} \frac{ h^{\frac{m+1}{2}} d\mu(\xi)}{d(\xi,z)^{n+\frac{m+1}{2}}}.
\end{multline}

\section{Pseudoanalytic continuation of holomorphic Besov functions\label{Cont_Aps}}
To prove two following theorems we use the Hardy inequality. Let $f$ be a positive function on $(0,\infty)$ and for $r\neq1$ define a function $F(x)$ as follows:
\begin{equation*}
    F(x) =  \left\{
  	\begin{array}{ll}
  	  \int\limits_0^{x} f(t) dt, & r>1; \\
 	   \int\limits_x^{\infty} f(t) dt, & r<1.
	\end{array}
\right.
\end{equation*}
Then for $1\leq p<\infty$ we have
\begin{equation} \label{ineq:hardy}
\int\limits_0^{\infty} x^{-r} F^p(x) dx <
\left(\frac{p}{\abs{r-1}}\right)^p \int\limits_0^{\infty} x^{-r}
(xf(x))^p dx.
\end{equation}
The proof and detailed discussion can be found in \cite{HLP52}.

\begin{theorem} \label{thm:Aps_pa}

Let $1\leq p,q\leq\infty,\ s>0$ and $f\in H^1(\Omega).$ Then $f\in
A^s_{pq}(\Omega)$ if and only if there exist a pseudoanalytic continuation $\f$ of function $f$ such that
\begin{equation} \label{eq:contcond}
\norm{S_p(\f,r)r^{1-1/q-s}}_{L^q[0,\eps]}<\infty.
\end{equation}

\end{theorem}

\begin{proof} The sufficiency is satisfied due to Corollary~\ref{cor:LocCont} of the Theorem~\ref{thm:LocCont} and the definition of Besov spaces.
To prove necessity assume that there exist a continuation $\f$ that satisfies  condition (\ref{eq:contcond}) and that $\supp{\f}\subset\O_\eps.$ 
Then by estimate (\ref{eq:PolyByCont}) we have
\[
	\omega_m(f,c\delta)_p \lesssim \norm{g}_{L^p(\dom)} + \norm{h}_{L^p(\dom)}
\]
for some $c>0,$ where
\begin{align*}
& g(z) = \int\limits_{d(\xi,z)<2\delta} \frac{ 
\abs{\bar{\partial} \f(\xi)} }{ d(\xi,z)^n } d\mu(\xi);\ 
h(z) = \int\limits_{d(\xi,z)>\delta}\frac{\abs{\bar{\partial} \f(\xi)}
\delta^{\frac{m+1}{2}}}{d(\xi,z)^{n+\frac{m+1}{2}}}
d\mu(\xi),\ z\in\dom.
\end{align*}

We introduce the following functions
\begin{align*}
 g_r(z) = \int\limits_{\substack{\xi\in\dom_r\\ d(\xi,z)<2\delta}} \frac{ \abs{\bar{\partial} \f(\xi)} }{ d(\xi,z)^n } d\sigma_r(\xi);\ 
 h_r(z) = \int\limits_{\substack{\xi\in\dom_r\\ d(\xi,z)>\delta}}\frac{\abs{\bar{\partial} \f(\xi)}
\delta^{\frac{m+1}{2}}}{d(\xi,z)^{n+\frac{m+1}{2}}}
d\sigma_r(\xi), z\in\dom.
\end{align*}
Then
\[
 \omega_m(f,c\delta)_p \lesssim \int\limits_{0}^{2\delta}\norm{g_r}_{L^p(\dom)} dr + \int\limits_{0}^{1}\norm{h_r}_{L^p(\dom)} dr. 
\]
We will prove that
\begin{align*}
& \norm{g_r}_{L^p(\dom)} \lesssim S_p(\f,r) \log{\frac{\delta}{r}},\quad 0<r<2\delta;\\
& \norm{h_r}_{L^p(\dom)} \lesssim S_p(\f,r) ,\quad 0<r<\delta;\quad
\norm{h_r}_{L^p(\dom)} \lesssim S_p(\f,r)
\frac{\delta^{\frac{m+1}{2}}}{r^{\frac{m+1}{2}}} ,\quad r>\delta.
\end{align*}
By Riesz-Thorin interpolation theorem it is enough to prove these estimates for $p=1$ and $p=\infty.$ By Lemma~\ref{LerayEst} for $r>\delta$ we have
\[
\norm{h_r}_{L^1(\dom)} = \int\limits_{\dom_r} \abs{\bar{\partial} \f(\xi)} d\sigma(\xi) \int\limits_{\substack{z\in\dom\\ d(\xi,z)>\delta}}\frac{\delta^{\frac{m+1}{2}}d\sigma(z)}{d(\xi,z)^{n+\frac{m+1}{2}}} \lesssim S_1(\f,r) \frac{\delta^{\frac{m+1}{2}}}{r^{\frac{m+1}{2}}},
\]
 for  $p=1$ and for $p=\infty$
\[
\abs{h_r(z) } \lesssim S_\infty(\f,r) \int\limits_{\substack{z\in\dom\\ d(\xi,z)>\delta}}\frac{\delta^{\frac{m+1}{2}}d\sigma(z)}{d(\xi,z)^{n+\frac{m+1}{2}}} \lesssim  S_\infty(\f,r)   \frac{\delta^{\frac{m+1}{2}}}{r^{\frac{m+1}{2}}},\ z\in\dom,
\]
which concludes the desired estimate. Second estimate of function $h_r$ and estimate of function $g_r$
can be proven analogously.

Thus for every $\eps>0,\ m\in\N$ 
\begin{equation} \label{eq:om_constr}
 \omega_m(f,c\delta)_p \lesssim \int\limits_{0}^{2\delta} S_p(\f,r) \frac{\delta^\eps}{r^\eps} dr + \delta^{\frac{m+1}{2}}\int\limits_{\delta}^{1} S_p(\f,r) \frac{dr}{r^{\frac{m+1}{2}}}.
\end{equation}
Assume that $1\leq q<\infty.$ We will use Hardy inequality to prove an estimate
\begin{equation*}
 \int\limits_0^\eps\omega_m(f,\delta)_p^q \delta^{-qs-1}d\delta < \infty.
\end{equation*}
Consider the first term in
(\ref{eq:om_constr}) letting $\eps=\frac{s}{2}$ and
$
F_1(\delta) = \int\limits_0^\delta S_p(\f,r) \frac{dr}{r^\eps}.
$ 
Then applying Hardy inequality (\ref{ineq:hardy}) with $r=1+sq>1$ we have
\begin{multline*}
\int\limits_0^1 (\delta^\eps F_1(\delta))^q \delta^{-1-sq} d\delta \lesssim \int\limits_0^1 r^{-1-sq+\eps q} \left(S_p(\f,r) r^{-\eps}r \right)^q dr\\ = \int\limits_0^1 \left(\frac{S_p(\f,r)}{r^{s-1}}\right)^q \frac{dr}{r}<\infty.
\end{multline*}
Analogously we can estimate the second term in
(\ref{eq:om_constr}) letting $s<(m+1)/2.$

For $q=\infty$ we have $S_p(\f,r)\lesssim r^{s-1}$ and
\[
\omega_m(f,c\delta) \lesssim \delta^\eps\int_0^{2\delta}r^{s-1-\eps}dr+\delta^{\frac{m+1}{2}}\int_\delta^\eps r^{s-1-\frac{m+1}{2}}dr
\lesssim \delta^s.
\]
This concludes the proof of the theorem.
\end{proof} 

\section{Main theorem.}

\begin{theorem} \label{thm:ConstrCond}

Let $1\leq p,q\leq\infty,\ s>0,$ and $f\in H^p(\dom).$ Then $f\in
A^s_{pq}(\Omega)$ if and only if $\left\{2^{ms}E_m(f)_{p}\right\}_{m=1}^{\infty}\in l^q.$
\end{theorem}

\begin{proof}

Assume that $\left\{2^{ms}E_{2^m}(f)_{p}\right\}_{m=1}^{\infty}\in l^q.$
By Corollary~\ref{cor:ContGlob} of Theorem~\ref{thm:ContGlob} we have a pseudoanalytic continuation such that
$$S_p(\f,r)\lesssim 2^mE_{2^m}(f)_p. $$
Then
\begin{equation}\label{ConstrCond1}
 \norm{{S_p(\f,r)}{r^{-s-1/q+1}}}_{L^q(0,\eps)} \lesssim \norm{ \{ {2^{ms}} E_{2^m}(f)_{p}\}_{m=1}^\infty}_{l^q}<\infty,
\end{equation}
and by Theorem~\ref{thm:Aps_pa} $f\in A^s_{pq}(\Omega).$

Conversely, let $f\in A^s_{pq}(\Omega).$ Then there exist a pseudoanalytic continuation $\f$ of function $f$ satisfying estimate (\ref{eq:contcond}). We will prove that
\begin{equation}\label{eq:Em}
 E_m(f)_p \lesssim \int\limits_0^{1/m} S_p(\f,r) \frac{dr}{(mr)^{\eps}} dr +\int\limits_{1/m}^\infty S_p(\f,r) \frac{dr}{(mr)^{\alpha}},
\end{equation}
where the parameter $\alpha$ can be chosen arbitrary great and $\eps>0$ is small enough. We construct polynomials approximating a function $f$ using an approximation of a kernel $K(\xi,z)$ from Lemma~\ref{lm:KGlob}. Let 
\[
  P_m(z) = \int_{\CO}
\bar{\partial} \f(\xi)\wedge\omega(\xi) K^{glob}_m(\xi,z).
\]
 Then
\begin{multline*}
 \abs{f(z)-P_m(z)} \lesssim \int_{\CO} \abs{\bar{\partial}\f(\xi)} \abs{K(\xi,z) - K^{glob}_m(\xi,z) } d\mu(\xi)\\ 
 \lesssim  \int\limits_{d(\xi,z)<1/m} \frac{\abs{\bar{\partial} \f(\xi)}}{d(\xi,z)^n}d\mu(\xi) +  \int\limits_{\substack{d(\xi,z)<1/m,\\ \rho(\xi)<1/m}} \frac{\abs{\bar{\partial}\f(\xi)}}{m^\alpha d(\xi,z)^{n+\alpha}}d\mu(\xi)\\ + \int\limits_{\rho(\xi)>1/m}
\frac{\abs{\bar{\partial}\f(\xi)}}{m^\alpha d(\xi,z)^{n+\alpha}}d\mu(\xi)
 = U_m(z) + V_m(z) + W_m(z).
\end{multline*}

Using notations of Theorem~\ref{thm:Aps_pa} we get
\[
	U_m(z)\lesssim \int\limits_0^{1/m} dr \int_{\dom_r} \frac{\abs{\bar{\partial}\f(\xi)}d\sigma_r(\xi)}{d(\xi,z)^n} =\int\limits_0^{1/m} g_r(z) dr.
\]
Then $\norm{g_r}_{L^p(\dom)} \lesssim
S_p(\f,r)\log{\frac{2}{mr}},\ 0<r<1/m,$ and by Minkovsky integral inequality
\[
	\norm{U_m}_{L^p(\dom)} \lesssim \int\limits_0^{1/m} S_p(\f,r)\log{\frac{2}{mr}}\ dr\lesssim\int\limits_0^{1/m} S_p(\f,r)\frac{dr}{(mr)^\eps}
\]
for every $\eps>0.$ Analogously
\[
	\norm{V_m}_{L^p(\dom)} \lesssim \int\limits_0^{1/m} S_p(\f,r)\ dr ;\
	\norm{W_m}_{L^p(\dom)} \lesssim \int\limits_{1/m}^\infty S_p(\f,r)
\frac{dr}{(mr)^{\alpha}},
\]
which concludes the proof of (\ref{eq:Em}).

As in the proof of Theorem~\ref{thm:Aps_pa} the final estimate for $1\leq q<\infty$ follows from Hardy inequality (\ref{ineq:hardy}).  
Let $F(t) =
\int\limits_0^{t} S_p(\f,r)\frac{dr}{r^\eps}$ and $\eps<s.$ Then by monotonicity of function $F$
\begin{multline*}
 \sum\limits_{m=1}^\infty 2^{msq}\norm{U_m}_{L^p(\dom)}\lesssim \sum\limits_{m=1}^\infty 2^{msq} \left(\int\limits_0^{2^{-m}} S_p(\f,r)\frac{dr}{(2^mr)^\eps}\right)^q\\
 =\sum\limits_{m=1}^\infty 2^{m(s-\eps)q}F(2^{-m}) \lesssim \int\limits_{0}^{\infty} t^{(\eps-s)q-1} F(t)^q dt  \\
\lesssim\int\limits_{0}^{\infty} r^{(\eps-s)q-1} ( S_p(\f,r) r^{1-\eps})^q
dr = \int\limits_{0}^{\infty} S_p(\f,r)^q r^{-q(s-1)-1} dr<\infty.
\end{multline*}
Analougously we get a second estimate for $V_m$ and $W_m.$

For $q=\infty$ we have $S_p(\f,r)\lesssim r^{s-1}$ and $E_m(f)_p\lesssim m^{-s}$ by estimate (\ref{eq:Em}).
This concludes the proof of the theorem.
\end{proof}

\end{document}